\documentclass[12pt,letterpaper]{amsart}
\usepackage{euler, epic,eepic,latexsym, amssymb, amscd, amsfonts, xypic, url, color, epsfig}
 \newlength{\baseunit}               
 \newcount{\numlines}                
\newtheorem{thmintro}{Theorem}
\newtheorem{tm}{Theorem}[section]
\newtheorem{pr}[tm]{Proposition}
\newtheorem{lm}[tm]{Lemma}
\newtheorem{co}[tm]{Corollary}

\newtheorem{exa}[tm]{Example}

\newcommand{\cat}[1]{{\mathbf{#1}}}
\newcommand{\ho}{\operatorname{ho}}
\newcommand{\pro}{\operatorname{pro}}

\newcommand{\Aone}{{\mathbb{A}^{\!1}}}

\newcommand{\proj}{\mathbb{P}}

\DeclareMathOperator*{\colim}{colim}

\newcommand{\rH}{\operatorname{H}}

\newcommand{\Z}{\mathbb{Z}}

\newcommand{\Zhat}{\hat{\mathbb{Z}}}
\newcommand{\Q}{\mathbb{Q}}

\newcommand{\kbar}{\overline{k}}

\newcommand{\C}{\mathbb{C}}
\newcommand{\G}{\mathbb{G}}

\newcommand{\hatZ}{\mathbb{Z}^{\wedge}}

\newcommand{\Hom}{\operatorname{Hom}} 

\newcommand{\Gal}{\operatorname{Gal}}

\newcommand{\Grp}{\cat{Grp}}

\newcommand{\Ab}{\cat{Ab}}
\newcommand{\prosSet}{\cat{pro-sSet}}
\newcommand{\sSet}{\cat{sSet}}

\newcommand{\Sm}{\cat{Sm}} 
\newcommand{\spaces}[1]{\cat{sPre}(\Sm_{#1})} 
\newcommand{\Spaces}{\cat{sPre}(\Sm_k)} 

\newcommand{\pGout}{\pro-\Grp^{\out}_{\pro-G^{\wedge}}}

\newcommand{\ott}{\overline{\theta}_{3}}

\newcommand{\Et}{\operatorname{Et}}
\newcommand{\LEt}{\operatorname{L Et}}
\newcommand{\Th}{\operatorname{Th}}

\newcommand{\Spec}{\operatorname{Spec}}

\newcommand{\Ker}{\operatorname{Ker}}

\newcommand{\out}{\mathrm{out}}

\newcommand{\et}{{\text{\'et}}}
\newcommand{\ab}{{\rm ab}}

\newcommand{\calN}{{ \mathcal N}}

\newcommand{\hidden}[1]{\footnote{Hidden:  #1}}
\renewcommand{\hidden}[1]{}

\begin{document}
\pagestyle{plain}
\title{Desuspensions of $S^1 \wedge (\proj_{\Q}^1 - \{0,1,\infty \})$}

\author{Kirsten Wickelgren}
\address{School of Mathematics, Georgia Institute of Technology, Atlanta~GA}
\date{February 2015.}
\begin{abstract}
We use the Galois action on $\pi_1^{\et}(\proj_{\overline{Q}}^1 - \{0,1,\infty \})$ to show that the homotopy equivalence $S^1 \wedge (\G_{m,\Q} \vee \G_{m,\Q}) \cong S^1 \wedge (\proj_{\Q}^1 - \{0,1,\infty \}) $ coming from purity does not desuspend to a map $\G_{m,\Q} \vee \G_{m,\Q} \to \proj_{\Q}^1 - \{0,1,\infty \}$.

\end{abstract}
\maketitle

{\parskip=12pt 

\section{Introduction} 

The \'etale fundamental group of the scheme $\proj^1 - \{0,1,\infty\}$ contains interesting arithmetic \cite{Deligne_Galois_groups} \cite{Ihara_braids}. By viewing schemes as objects in the $\Aone$-homotopy category of Morel-Voevodsky \cite{MV}, we may form the simplicial suspension $\Sigma X =S^1 \wedge X$ of a pointed scheme $X$, and the wedge product of two pointed schemes. After one simplicial suspension, $\proj^1 - \{0,1,\infty\}$ and the wedge $\G_m \vee \G_m$ of two copies of $\G_m$ become canonically $\Aone$-equivalent by the purity theorem \cite[Theorem 2.23]{MV}, and this is given in Proposition \ref{canonical_iso_SigmaP1minus3points}. This paper uses calculations of Anderson, Coleman, Ihara and collaborators \cite{Anderson_hyperadelic_gamma} \cite{Coleman_Gauss_sum} \cite{Ihara_Braids_Gal_grps} \cite{IKY}  on the \'etale fundamental group of $\proj^1 - \{0,1,\infty\}$ to show that this $\Aone$-equivalence does not desuspend, in the sense that there does not exist a map $\G_{m,\Q} \vee \G_{m,\Q} \to \proj^1_{\Q} - \{0,1,\infty \}$ whose suspension is equivalent to the map coming from purity. This can be summarized by the statement that the Galois action on the \'etale fundamental group of $\proj^1 - \{0,1,\infty\}$ is an obstruction to desuspension. 

There are topological obstructions to desuspension coming from James-Hopf maps, and they generalize to the setting of $\Aone$-homotopy theory \cite{SEHPWickelgrenWilliams} \cite{AFWW_SimpSuspSeq} as was known to Morel. They do not a priori have a relationship with the Galois action on $\pi_1^{\et}(\proj^1_{\overline{\Q}} - \{0,1,\infty\})$. This paper is motivated by the contrast between the systematic tools from algebraic topology available to obstruct desuspension and the arithmetic of $\pi_1^{\et}(\proj^1_{\Q} - \{0,1,\infty\})$ which shows that such a desuspension does not exist.

The results of this paper are as follows. Let $\Sm_k$ denote the full subcategory of finite type schemes over a characteristic $0$ field $k$ whose objects are smooth schemes. Let $\Spaces$ denote presheaves of simplicial sets on $\Sm_k$. $\Spaces$ has the structure of a simplicial model category in several ways.  Let $\ho_{\Aone} \Spaces$ denote the homotopy category of the $\Aone$-local, projective \'etale (respectively Nisnevich) model structure on $\Spaces$. The results of this paper hold with either the Nisnevich or \'etale Grothendieck topology, and we will use the same notation for either.\hidden{ Once the choice of the \'etale of Nisnevich topology is made, there are many equivalent ways to describe $\ho_{\Aone} \Spaces$. For example, $\ho_{\Aone} \Spaces$ is equivalent to the homotopy category of the $\Aone$-local, injective (or flasque) \'etale  (respectively Nisnevich) model structure on $\Spaces$.} $\ho_{\Aone} \Spaces$ is formed by formally inverting $\Aone$-weak equivalences and local \'etale (respectively Nisnevich) equivalences. In particular, in the injective Nisnevich case, $\ho_{\Aone} \Spaces$ is the $\Aone$-homotopy category of \cite{MV}. 

In Section \ref{stab_iso_section}, we give the canonical isomorphism in $\ho_{\Aone} \Spaces$ discussed above between the unreduced simplicial suspension of $\proj_k^1 - \{0,1,\infty \}$ and the simplicial suspension $\Sigma(\G_{m} \vee \G_{m}) = S^1 \wedge (\G_{m} \vee \G_{m})$. The reduced and unreduced simplicial suspensions are canonically isomorphic in $\ho_{\Aone} \Spaces$, although to form the reduced simplicial suspension, a base point is required. Let $\wp:  \Sigma (\G_{m} \vee G_{m}) \to \Sigma(\proj_k^1 - \{0,1,\infty \}) $ denote any of the maps in $\ho_{\Aone} \Spaces$ resulting from choosing a base point in $\proj_k^1 - \{0,1,\infty \}$. 

\begin{thmintro}\label{Intro_main_thm}
Let $k$ be a finite extension of $\Q$ not containing a square root of $2$. There is no morphism $\G_{m,k} \vee \G_{m,k} \to \proj_{k}^1 - \{0,1,\infty \}$ in $\ho_{\Aone} \Spaces$ whose simplicial suspension is $\wp$.
\end{thmintro}

Theorem \ref{Intro_main_thm} is proved as Theorem \ref{no_desuspension}. The proof uses an \'etale realization, following ideas of Artin-Mazur, Friedlander, Isaksen, Quick, and Schmidt. The needed results are collected in Section \ref{Sec_Equivariant_etale_homotopy_type}.

\subsection{Acknowledgements} I wish to thank Michael Hopkins and Ben Williams for useful discussions, and also Emily Riehl. This work is supported by NSF grant DMS-1406380.

\section{Equivariant \'Etale homotopy type}\label{Sec_Equivariant_etale_homotopy_type}

 Let $\kbar$ be an algebraic closure of $k$, and $G=\Gal(\kbar/k)$ denote the absolute Galois group of $k$. Let $\hatZ$ denote the profinite completion of $\Z$. Let $\chi: G \to (\hatZ)^{\ast}$ denote the cyclotomic character. Let $\hatZ(n)$ denote $\hatZ$ with the continuous $G$-action where $g$ in $G$ acts by multiplication by $\chi^n(g)$. For a pro-group $J =\{J_{\alpha}\}_{\alpha \in A}$, let $B P$ denote the pro-simplicial set $\{B J_{\alpha}\}_{\alpha \in A}$ given as the inverse system of the classifying spaces $B J_{\alpha}$ of the groups $J_{\alpha}$. For a group $J$, let $\pro-J^{\wedge}$ denote the pro-group given as the inverse system of the finite quotients of $J$. 

Let $\Et: \Spaces \to \prosSet$ denote the \'etale homotopy type of \cite[Definition 1]{Isaksen}. This \'etale homotopy type is built using the \'etale homotopy and topological types of \cite{ArtinMazur} and \cite{Friedlander}. An alternate extension of the \'etale homotopy type to the $\Aone$-homotopy category of schemes was constructed independently by Alexander Schmidt \cite{Schmidt_ES} \cite{Schmidt_AG}. By \cite[Theorem 2]{Isaksen}, $\Et$ is a left Quillen functor with respect to the \'etale or Nisnevich local projective model structure on $\Spaces$ and the model structure on $\prosSet$ given in \cite{Isaksen_model_protop}. This model structure on $\prosSet$ is such that weak equivalences are the maps $f:X \to Y$ inducing an isomorphism of pro-sets $\pi_0(X) \to \pi_0(Y)$ and inducing an isomorphism from the local systems associated to $\pi_n (X)$ to the pull-back of the local system on $Y$ coming from $\pi_n(Y)$. The cofibrations are maps isomorphic to levelwise cofibrations. See \cite[Definition 6.1, 6.2]{Isaksen}. The proof that $\Et$ is a left Quillen functor uses \cite[Prop 2.3]{DuggerUHT}. Let $\LEt: \ho \Spaces \to \ho \prosSet$ denote the corresponding homotopy invariant derived functor. By \cite[Corollary 4]{Isaksen}, the functor $\LEt$ agrees with the \'etale homotopy type for a scheme in $\Sm_k$. The standard calculation of the \'etale homotopy type of $\Spec k$ then gives $\LEt(\Spec k)\cong B (\pro-G^{\wedge})$.

\subsection{Fundamental groupoids} Let $X$ be a scheme over $k$. Let $\Pi_1^{\et} X$ denote the \'etale fundamental groupoid of $X$ whose objects are geometric points and morphisms are \'etale paths, i.e., natural transformations between the associated fiber functors \cite[V 7]{sga1}. The morphisms are topologized with the natural profinite topology.

For a simplicial set $X$, let $\Pi_1$ denote the fundamental groupoid. For a pro-simplicial set $X =\{ X_{\alpha} \}_{\alpha \in A}$, let $\Pi_1 X$ denote the pro-groupoid $\{ \Pi_1 X_{\alpha} \}_{\alpha \in A}$. 

\subsection{Base points and fundamental groups}
Given a map $\ast \to X = \{ X_{\alpha} \}_{\alpha \in A}$ in $\prosSet$, each simplicial set $X_{\alpha}$ has a base-point coming from the definition of morphisms $$\Hom(\ast, X) = \varprojlim_{\alpha \in A} \varinjlim \Hom(\ast, X_{\alpha}) \cong \varprojlim_{\alpha \in A}  \Hom(\ast, X_{\alpha})$$ in the pro-category. Thus there is a distinguished object in each $\Pi_1 X_{\alpha}$. The endomorphisms of this object fit into a pro-group, defined to be the fundamental group $\pi_1(X)$ of $\ast \to X$. \hidden{Alternatively, the  fundamental group $\pi_1(X)$ is defined to be the pro-group $ \{ \pi_1(X_{\alpha},\ast) \}_{\alpha \in A}$. }

 For $X$ in $\Spaces$, use the notation $\Pi_1(X)$ for $\Pi_1 \LEt X$. The \'etale homotopy type $\Et$ takes a scheme $X$ equipped with a geometric point $\Spec \Omega \to X$  to a pointed pro-simplicial set because $\Et (\Spec \Omega) \cong \ast$ for $\Omega$ an algebraically closed field. The resulting $\pi_1$ is independent of the choice of isomorphism $\Et (\Spec \Omega) \cong \ast$. 

Let $\overline{x}:\Spec \Omega \to X$ be a geometric point. Replace $\Omega$ by the subfield of $\Omega$ given by the algebraic closure of the residue field of the point of $X$ in the image of $\overline{x}$. This replacement has finite transcendence degree, and therefore $\Spec \Omega$  is an essentially smooth $k$-scheme in the sense of \cite[vi]{morel2012}, i.e. a Noetherian scheme which is an inverse limit of a left filtering system $\{ \Omega_{\alpha}\}_{\alpha \in A}$ of smooth $k$-schemes with \'etale affine transition morphisms. Since $X$ is assumed to be finite type over $k$, the map $\overline{x}$ is determined by the images of finitely many functions on an open subset of $X$, and thus determines an element of $\varinjlim X(\Omega_{\alpha})$. As in \cite{morel2012}, given $X \in \Spaces$ and an essentially smooth $k$-scheme $\{Y_{\alpha} \}_{\alpha \in A}$, define $X(Y) = \varinjlim X(Y_{\alpha})$, and call $X(Y)$ the set of $Y$ points of $X$. For $X$ in $\Spaces$, a {\em geometric point} of $X$ indicates an element of $X(\Spec \Omega)$, where $\Omega$ is an algebraically closed field of finite transcendence degree over $k$. Note that for $X \in \Spaces$, a geometric point $\overline{x} \in X(\Spec \Omega)$ induces a map $\Et (\Spec \Omega) \to \LEt X$.

Let $k-\Sm^+$ denote the following category. The objects of $k-\Sm^+$ are pairs $(X, \overline{x})$, where $X$ is a smooth $k$-scheme and $\overline{x}$ is a geometric point of $X$ equipped with a path between its image under $X \to \Spec k$ and the geometric point $\Spec \kbar \to \Spec k$ of $k$. The morphisms $(X, \overline{x}) \to (Y, \overline{y})$ of $k-\Sm^+$ are the morphisms $X \to Y$ in $k-\Sm$. There is no requirement that $\overline{x}$ is taken to $\overline{y}$. Let $k-\Sm^+_c$ denote the full subcategory of $k-\Sm^+$ of objects such that $X$ is connected. Let $\Grp_{G}^{\out}$ denote the category of topological groups over $G$ and outer homomorphisms, i.e., the objects of $\Grp_{G}^{\out}$ are morphisms $\pi \to G $ and the morphisms from $\pi \to G $ to $\pi' \to G$ is the set of equivalence classes of morphisms $\pi \to \pi'$ such that the two morphisms $\pi \to G$ coming from the diagram $$\xymatrix{\pi \ar[rd] \ar[rr]&& \pi' \ar[dl]\\ &G&} $$ differ by an inner automorphism of $G$ and where two such morphims $f,f':\pi \to \pi'$ are considered equivalent if there exists $\gamma \pi'$ such that $f'(x) = \gamma f(x) \gamma^{-1}$ for all $x$ in $\pi$. 

Given a morphism $\Pi \to \Pi'$ of pro-groupoids and commutative diagram $$\xymatrix{ \Pi \ar[r] & \Pi' \\ \ast \ar[r] \ar[u]& \ast \ar[u]},$$ there is an associated morphism of fundamental pro-groups $\pi \to \pi'$, where $\pi$ is the endomorphims in $\Pi$ of the distinguished object and similarly for $\pi'$. Given two maps $x_1,x_2: \ast \to \Pi'$ and a choice of morphism from $x_1$ to $x_2$ we obtain an isomorphism between the fundamental group based at $x_1$ and the fundamental group based at $x_2$. A different choice of path changes the isomorphism by an inner isomorphism. To a map between objects of  $k-\Sm^+_c$, we may therefore associate an outer homomorphism. We claim that this defines a functor $\pi_1^{\et}: k-\Sm^+_c \to \Grp_{G}^{\out}$. To see this, note that for the object $(X, \overline{x})$ of $k-\Sm^+_c$, the path between the image of $\overline{x}$ under $X \to \Spec k$ and the geometric point $\Spec \kbar \to \Spec k$ produces a morphism $\pi_1^{\et}(X, \overline{x}) \to G$. Given $(X, \overline{x}) \to (Y, \overline{y})$, the induced outer homomorphism $\pi_1^{\et}(X, \overline{x}) \to \pi_1^{\et}(Y, \overline{y})$ respects the maps to $G$ up to inner automorphism because $X \to Y$ respects the maps to $\Spec k$. This shows that $\pi_1^{\et}$ determines the claimed functor.

We need a fundamental group on pointed objects of $\Spaces$ with similar functoriality properties, so we introduce notation in this context analogous to the above. Let $\Spaces^+$ denote the category whose objects are pairs $(X, \overline{x})$, where $X$ is in $\Spaces$ and $\overline{x}$ is a geometric point of $X$ whose image in the set of geometric points of $\Spec k$ has a chosen path to $\Spec \kbar \to \Spec k$, and whose morphisms $(X, \overline{x}) \to (Y, \overline{y})$ are the morphisms $X \to Y$ in $\Spaces$. There is again no requirement that $\overline{x}$ is taken to $\overline{y}$. Define $\ho_{\Aone} \Spaces^+$ similarly, i.e., the morphisms $(X, \overline{x}) \to (Y, \overline{y})$ are the morphisms $X \to Y$ in $\ho_{\Aone}\Spaces$. Let $\Spaces^+_c$ denote the full-subcategory of $\Spaces^+$ on objects such that $\LEt X$ is connected. Similarly define $\ho_{\Aone} \Spaces^+_c$ to be the full-subcategory of $\ho_{\Aone} \Spaces^+$ on objects such that $\LEt X$ is connected. Let $\pGout$ denote the category of pro-groups over $\pro-G^{\wedge}$ and outer homomorphisms.

For $(X, \overline{x})$ in $\Spaces^+$, define $\pi_1(X, \overline{x}: \Spec \Omega \to X)$ to be $\pi_1$ of the pointed pro-simplicial set $\ast \cong \Et (\Spec \Omega) \to \LEt X$. By the same argument as above, $\pi_1$ defines a functor $\pi_1: \Spaces^+_c \to \pGout$.

\subsection{Homotopy invariant functors} 

\begin{pr}\label{etpi_on_spaces}
The functor $\pi_1: \Spaces^+_c \to \pGout$ factors through $\ho_{\Aone} \Spaces^+_c$. Furthermore, the diagram $$\xymatrix{ k-\Sm^+\ar[d] \ar[rr]^{\pi_1^{\et}}&&\Grp_{G}^{\out} \\ \ho_{\Aone} \Spaces_c^+ \ar[rr]^{\pi_1}&&\pGout \ar[u]^{\varprojlim} }$$ commutes up to isomorphism. 
\end{pr}

\begin{proof}
For a scheme $X$, let $\mathcal{X}$ in $\Spaces$ denote the corresponding sheaf. There is a natural isomorphism $\pi_1^{\et} (X) \cong \varprojlim \pi_1 \LEt \mathcal{X}$ for every smooth scheme $X$ over $k$ equipped with a geometric point because both sides classify finite \'etale covers of $X$. For the left hand side, this is immediate. For the right hand side, this follows from \cite[Prop 5.6]{Friedlander} and \cite[11.1]{ArtinMazur}.

To show that $\pi_1$ factors through $\ho_{\Aone} \Spaces^+_c$, it suffices to show that the functor $\Pi_1$ from spaces to pro-groupoids factors through $\ho_{\Aone} \Spaces$. Since $\Pi_1 = \Pi_1 \LEt$, we know that $\Pi_1$ factors through the homotopy category of the \'etale (respectively Nisnevich) local projective model structure. To show the factorization through $\ho_{\Aone} \Spaces$, it is thus sufficient to show that $X \times \mathbb{A}^1 \to X$ is sent to an isomorphism for all schemes $X$. This follows from the analogous claim on \'etale fundamental groups, which is true in characteristic 0. (One can see that $\pi_1^{\et}(X \times \mathbb{A}^1) \to \pi_1^{\et} X$ is an isomorphism in characteristic $0$ by combining \cite[IX Th\'eor\`eme 6.1]{sga1} with the analogous result over $\kbar$. Over $\kbar$, the map is an isomorphism by invariance of $\pi_1^{\et}$ under algebraically closed extensions of fields \cite[XIII Proposition 4.6]{sga1} and comparison with the topological fundamental group \cite[XII Corollaire 5.2]{sga1}.)
\end{proof}

\begin{exa}\label{LEtGmkveeGmk}
We compute $\pi_1(\G_{m,k} \vee \G_{m,k}, \ast ) \to G$. The map $\ast \to \G_{m,k}$ corresponding to the point $1$ is a flasque cofibration because it is the push-out product of itself and $\partial \Delta^0 \to \Delta^0$. See \cite[Definition 3.2]{Isaksen_flasque}. Since representable presheaves are projective cofibrant\hidden{\cite[proof of Corollary 4]{Isaksen}}, $\ast$ and $\G_{m,k}$ are projective cofibrant, whence also flasque cofibrant. It follows that $\G_{m,k} \vee \G_{m,k}$ is a homotopy colimit in the flasque model structure. Since $\Et$ is a left Quillen functor on the Nisnevich (or \'etale) local flasque model structure \cite[Theorem 3.4]{Quick_pro-finite_htpy} and since there is a weak equivalence between $\Et$ derived with respect to the local  flasque model structure, and $\LEt$, which denotes $\Et$ derived with respect to the projective local model structure, it follows that \begin{equation*}\xymatrix{\LEt (\Spec k) \ar[r] \ar[d] & \LEt(\G_{m,k}) \ar[d] \\  \LEt(\G_{m,k}) \ar[r] & \LEt (\G_{m,k} \vee \G_{m,k}) } \end{equation*} is a homotopy push-out square.  \hidden{To see that that there is a weak equivalence between $\Et$ derived with respect to the local  flasque model structure, and $\LEt$, note that the flasque derived and the projective derived $\Et$ both map to $\Et$ applied to the projective cofibrant replacement of the flasque cofibrant replacement. Since projective cofibrations are also flasque cofibrations  by \cite[Theorem 3.7]{Isaksen_flasque}, these maps are both weak equivalences because $\Et$ preserves weak equivalences between flasque cofibrant objects by Ken Brown's lemma and the fact that $\Et$ is a left Quillen functor with respect to the flasque model structure. Alternatively, we could change $\LEt$ as to be the derived functor for the flasque model structure.}

Let $\{ \langle x, y \vert x^n=1= y^n\rangle \}_n$ denote the pro-group given as the inverse system over $n$ of the free product of $\Z/n$ with $\Z/n$ and transition maps induced by quotient maps $\Z/(nm) \to \Z/n$. Let $G$ act on $\langle x,y \vert x^n=1= y^n \rangle$ by \begin{equation*}g x = x^{\chi(g)} \quad \quad g y = y^{\chi(g)}.\end{equation*} Let $I$ denote the directed set consisting of pairs $(n,H)$ with $n$ a positive integer and $H$ a finite quotient of $G$ which acts on $k(\mu_n)$, i.e., $H$ is such that the fixed field of $\Ker(G \to H)$ contains the $n$th roots of unity in $k$. $I$ is defined so that there is a map $(n,H) \to (n', H')$ exactly when $H'$ is a quotient of $H$ and $n'$ is a quotient of $n$.

We claim that $\LEt (\G_{m,k} \vee \G_{m,k}) \to \LEt \Spec k$ can be identified with $$\{ B (\langle x, y \vert x^n=1= y^n\rangle \rtimes H)\}_{(n,H) \in I }\to B \pro G^{\wedge}.$$ Since $\LEt \G_m$ is the \'etale topological type, the map $\LEt \G_m \to \LEt \Spec k$ can be identified with $B \pro (\hatZ(1) \rtimes G )^{\wedge} \to B \pro G$. It thus suffices to show that \begin{equation}\label{wanted_hp}\xymatrix{ B \pro G \ar[r] \ar[d] & \ar[d]B \pro (\hatZ(1) \rtimes G )^{\wedge} \\ \ar[r] B \pro (\hatZ(1) \rtimes G )^{\wedge} &\{ B \langle x, y \vert x^n=1= y^n\rangle \rtimes H\}_{(n,H) \in I } }\end{equation} is a homotopy push-out square. Since $B \pro G \to B \pro (\hatZ(1) \rtimes G)^{\wedge}$ is a level-wise section of a level-wise fibration of simplicial sets, it is isomorphic to a level-wise monomorphism and is therefore a cofibration. Thus is suffices to show that \eqref{wanted_hp} is a push-out.

To see this, let $D$ and $F$ be finite groups, with actions of a finite group $C$. By Van-Kampen's theorem, \begin{equation}\label{DFpush-out}\xymatrix{\ast \ar[d] \ar[r]& \ar[d]BF \\ BD \ar[r] & B(D \ast F)}\end{equation} is a push-out and a homotopy push-out, where $D \ast F$ denotes the free product of $D$ and $F$. Let $EC$ denote a universal cover of $BC$.  Applying $(-)\times_G EC$ to \eqref{DFpush-out} produces another push-out and homotopy push-out. It follows that \eqref{wanted_hp} is a push-out, as claimed. 

Thus $\pi_1(\G_{m,k} \vee \G_{m,k}, \ast ) \to G$ can be identified with the map $$\{ \langle x, y \vert x^n=1= y^n\rangle \rtimes H\}_{(n,H) \in I } \to \pro-G^{\wedge}.$$

Define $\pi'$ to be the free profinite group on two generators $\pi' = \langle x,y\rangle^{\wedge}$, and let $G$ act on $\pi'$ by \begin{equation}\label{G-action_pi'} g x = x^{\chi(g)} \quad \quad g y = y^{\chi(g)} .\end{equation}

\begin{lm}\label{From_pi'}
Any morphism in $\pGout$ from $\{ \langle x, y \vert x^n=1= y^n\rangle \rtimes H\}_{(n,H) \in I }$ to an inverse system of finite groups factors through $\pro-(\pi' \rtimes G)^{\wedge}$. 
\end{lm}

\begin{proof}
Let $\{J_{\alpha} \}_{\alpha \in A}$ be a pro-group with each $J_{\alpha}$ finite, and suppose $\{J_{\alpha} \}_{\alpha \in A}$ is equipped with a map $\{J_{\alpha} \}_{\alpha \in A} \to \pro-G^{\wedge} .$ Any morphism in $\pGout$ from $\{ \langle x, y \vert x^n=1= y^n\rangle \rtimes H\}_{(n,H) \in I }$ to $\{J_{\alpha} \}_{\alpha \in A}$ is represented by a morphism of pro-groups $$\{ \langle x, y \vert x^n=1= y^n\rangle \rtimes H\}_{(n,H) \in I } \to \{J_{\alpha} \}_{\alpha \in A}.$$ Such a morphism is an element of \begin{equation}\label{mor_pro-group}\varprojlim_{\alpha} \varinjlim_{(n,H)} \Hom(\langle x, y \vert x^n=1= y^n\rangle \rtimes H, J_{\alpha}).\end{equation} Since $J_{\alpha}$ is finite, this set is in natural bijection with  $$\varprojlim_{\alpha} \varinjlim_{(n,H)} \Hom((\langle x, y \vert x^n=1= y^n\rangle \rtimes H)^{\wedge}, J_{\alpha}).$$ Since there is a map $\langle x, y \rangle \to \langle x, y \vert x^n=1= y^n\rangle \rtimes H$ sending $x$ to $x \rtimes 1$ and $y$ to $y \rtimes 1$, there is an induced map $\pi' \to (\langle x, y \vert x^n=1= y^n\rangle \rtimes H)^{\wedge}$. Since this map is equivariant with respect to the quotient $G \to H$, there is an induced map $\pi' \rtimes G \to (\langle x, y \vert x^n=1= y^n\rangle \rtimes H)^{\wedge}$. By checking compatibility with the transition maps, it follows that the set of morphisms \eqref{mor_pro-group} is in natural bijection with $$\varprojlim_{\alpha}  \Hom(\pi' \rtimes G, J_{\alpha}).$$
\end{proof}
\end{exa}

Let $\rH^i(-,\Z/n) : \prosSet \to \Ab$ denote the functor which takes a pro-simplicial set $\{ X_\alpha\}_{\alpha \in I}$ to the abelian group $\colim_{\alpha \in I} \rH^i(X_{\alpha}, \Z/n)$, cf. \cite[\S 5]{Friedlander}. By \cite[Proposition 18.4]{Isaksen_model_protop}, $\rH^i(-\Z/n)$ passes to the homotopy category and determines a functor $$ \rH^i(-,\Z/n) : \ho \pro-\sSet \to \Ab.$$ Let $\rH^i_{\et}(-,\Z/n)$ denote the usual \'etale cohomology groups of a scheme with coefficients in $\Z/n$.

\begin{pr}\label{Hi_spaces_functor}
$\rH^i_{\et}(-,\Z/n): \Sm_k \to \pro \Ab$ factors through $\ho_{\Aone} \Spaces$. 
\end{pr}

\begin{proof}
By \cite[Corollary 4]{Isaksen}, $\LEt X$ is the \'etale topological type of \cite{Friedlander}. Thus by \cite[Proposition 5.9]{Friedlander}, $\rH^i(\LEt X, \Z/n)$ is naturally isomorphic to the \'etale cohomology $\rH^i_{\et}(X,\Z/n)$. Thus is suffices to show that $$\rH^i(\LEt (-), \Z/n): \ho \Spaces \to \Ab $$ factors through $\ho_{\Aone} \Spaces$.  Since the $\Aone$-model structure is obtained by left Bousfield localization at the maps $X \times \Aone \to X$ for every scheme $X$, it suffices to show that $\LEt$ takes $X \times \Aone \to X$ to an isomorphism of abelian groups. This is true by \cite[VI Corollary 4.20]{Milnebook}.
\end{proof}

Let $\rH^i(-,\Z/n)$ also denote the functor $\rH^i(-,\Z/n): \Spaces \to \Ab$ given by $\rH^i(\LEt (-), \Z/n)$. As in Proposition \ref{Hi_spaces_functor}, $\rH^i(-,\Z/n)$ factors through $\ho_{\Aone} \Spaces$. 

\begin{pr}\label{Hi+1Sigma=Hi}
There is a natural isomorphism of functors $\rH^i(-, \Z/n) \cong \rH^{i+1}(\Sigma(-), \Z/n)$.
\end{pr}

\begin{proof}
Let $X$ be an object of $\spaces{k}_{\ast}$. Since left derived functors commute with homotopy colimits, $$\xymatrix{\LEt X \ar[r] \ar[d] & \LEt \ast  \ar[d] \\ \LEt \ast \ar[r] & \LEt \Sigma X}$$ is a push-out square in the model structure of \cite{Isaksen_model_protop}.

In the model structure of \cite{Isaksen_model_protop}, the cofibrations are isomorphic to levelwise cofibrations of systems of simplicial sets of the same shape. Also, levelwise homotopy equivalences are weak equivalences. It follows that $$\xymatrix{ \{ X_\alpha\}_{\alpha \in A} \ar[r] \ar[d]&\ast \ar[d]\\ \ast \ar[r] & \{ \Sigma X_{\alpha}\}_{\alpha \in A}}$$ is a homotopy push-out. In particular, letting $\{ X_\alpha\}_{\alpha \in A} = \LEt X$, we have that $\LEt \Sigma X \cong \{ \Sigma X_{\alpha}\}_{\alpha \in A}$.  

The proposition then follows from the fact that in ordinary cohomology of simplicial sets, we have $\rH^{i+1}(\Sigma A_{\alpha}, \Z/n) \cong \rH^i(A_{\alpha}, \Z/n)$.
\end{proof}

\begin{pr}\label{H1=Hompi}
There is a natural isomorphism of functors $$\rH^1(-,\Z/n) \cong \Hom(\pi_1(-), \Z/n): \Spaces_c^+ \to \Ab.$$
\end{pr}\hidden{this $\pi_1$ is the pro-group.}

\begin{proof}
The claim is equivalent to exhibiting a natural isomorphism $$\rH^1(\LEt(-),\Z/n) \cong \Hom(\pi_1 \LEt(-), \Z/n).$$  There is natural isomorphism a natural isomorphism $$\rH^1(-,\Z/n) \cong \Hom(\pi_1(-), \Z/n): \ho \sSet \to \Ab.$$ This induces a natural isomorphism $$\rH^1(-,\Z/n) \cong \Hom(\pi_1(-), \Z/n): \ho \prosSet \to \Ab,$$ where $\Hom$ is the homomorphisms in the category of pro-groups. The desired natural isomorphism is obtained by pulling back by $\LEt$.
\end{proof}

\section{Stable isomorphism $\proj_k^1 - \{0,1,\infty \} \cong \G_m \vee \G_m$}\label{stab_iso_section}

Recall that the smash product $X \wedge Y$ of two pointed spaces $X$ and $Y$ is $X \wedge Y = X \times Y/ (\ast \times Y \cup X \times \ast)$, and that the wedge product $X \vee Y$ is the disjoint union with the two base points identified. These formulas hold sectionwise for simpllicial presheaves, e.g. $(X \vee Y)(U) = X(U) \vee Y(U)$. The simplicial suspension $\Sigma X$ of $X$ in $\Spaces$ is $\Sigma X = S^1 \wedge X$. Let $S$ denote the unreduced simplicial suspension, $S X = I \times X/ \sim$, where $I$ denotes the standard $1$-simplex, and $\sim$ denotes the equivalence relation defined $0 \times X \sim \ast_0$ and $1 \times X \sim \ast_1$, where $\ast_0$ and $\ast_1$ are two copies of the terminal object. 

\begin{pr}\label{canonical_iso_SigmaP1minus3points}
There is a canonical isomorphism $ \Sigma (\G_m \vee \G_m) \to S (\proj_k^1 - \{0,1,\infty \})$ in $\ho_{\Aone} \Spaces$ which sends $\ast_0$ to the base point.
\end{pr}

\begin{proof}
Let $i: Z \to \mathbb{A}^1_k$ be the reduced closed subscheme corresponding to the closed set $\{0,1\}$. Note that $\mathbb{A}^1_k- i(Z) \cong \proj_k^1 - \{0,1,\infty \}$ is an isomorphism of schemes. Let $\calN (i) \to Z$ denote the normal bundle to $i$, and let $\Th(\calN (i))$ denote the Thom space of $\calN (i)$, as in \cite[Definition 2.16]{MV}. By \cite[Theorem 2.23]{MV}, there is a canonical isomorphism \begin{equation}\label{A1dividedA1minus2points=Th}\Th(\calN (i))  \cong \mathbb{A}^1_k/ ( \mathbb{A}^1_k- i(Z)) \end{equation} in $\ho_{\Aone} \Spaces$. Since $\mathbb{A}^1_k- i(Z)  \to \mathbb{A}^1$ is an open immersion, it is a monomorphism and therefore a cofibration. It follows that $ \mathbb{A}^1_k/ ( \mathbb{A}^1_k- i(Z))$ is equivalent to the homotopy cofiber of $\mathbb{A}^1_k- i(Z) \to \mathbb{A}^1_k$. Since $\mathbb{A}^1_k \to \ast$ is a weak equivalence, this homotopy cofiber is equivalent to the homotopy cofiber of $\mathbb{A}^1_k- i(Z) \to \ast$. This later homotopy cofiber is equivalent to the unreduced suspension $S (\mathbb{A}^1_k- i(Z))$. 

Let $\mathcal{O}$ denote the trivial bundle of rank $1$ over $Z$, $\mathcal{O} = Z \times \mathbb{A}^1$, and let $\proj \calN (i) \to \proj (\calN (i) \oplus \mathcal{O})$ denote the closed embedding at infinity. The vector bundle $\calN (i)$ is trivial of rank $1$ over $Z$. Since there are no automorphisms of $\mathbb{A}^1$ fixing $0$ and $1$, we have a canonical coordinate $z$ with $\mathbb{A}^1= \Spec k[z]$. For any $p \in k$, the map $k[z]/\langle z-p\rangle \to \langle z-p\rangle/  \langle z-p\rangle^2$ sending $f(z)$ to $f(p)(z-p)$ gives a canonical trivialization of the normal bundle of the closed immersion $\Spec k[z]/\langle z-p \rangle \to \mathbb{A}^1$. This gives a trivialization of $\calN(i)$. We obtain a canonical isomorphism  $\proj (\calN (i) \oplus \mathcal{O})/ \proj \calN (i) \cong \proj^1 \vee \proj^1$.  Use the coordinate $z$ on $\mathbb{A}^1$ and $\G_m = \Spec k[z, \frac{1}{z}]$. The reasoning above gives an equivalence $S\G_{m,k} \cong \mathbb{A}^1/ \G_{m,k} \to \proj^1_k$. The natural map $S \G_{m,k} \to \Sigma \G_{m,k}$ is a sectionwise weak equivalence, and thus gives a canonical isomorphism in the homotopy category. This yields a canonical isomorphism $\Sigma( \G_{m,k} \vee \G_{m,k}) \cong \proj (\calN (i) \oplus \mathcal{O})/ \proj \calN (i)$ in $\ho_{\Aone} \Spaces$. By \cite[Proposition 2.17. 3.]{MV}, there is a canonical equivalence $\proj (\calN (i) \oplus \mathcal{O})/ \proj \calN (i) \to \Th(\calN (i))$. Combining with \eqref{A1dividedA1minus2points=Th} produces the desired  canonical isomorphism.
\end{proof}

\begin{co}\label{wp_co}
For any choice of base point of $\proj_k^1 - \{0,1,\infty \}$, there is a canonical isomorphism $\Sigma (\G_m \vee \G_m) \to \Sigma (\proj_k^1 - \{0,1,\infty \})$ in $\ho_{\Aone} \Spaces$.
\end{co}

\begin{proof}
Since the natural map from the unreduced suspension to the reduced is a sectionwise weak equivalence, the unreduced suspension is equivalent to the reduced in a canonical manner. Thus we have a canonical equivalence between $S(\proj_k^1 - \{0,1,\infty \})$ and $\Sigma (\proj_k^1 - \{0,1,\infty \})$, so the corollary follows from Proposition \ref{canonical_iso_SigmaP1minus3points}. 
\end{proof}

Let $c_i : \G_m \vee \G_m \to \G_m$ for $i=1$ (respectively $i=2$) be the map which crushes the first (respectively second) summand of $\G_m$. Let $a_1: \proj_k^1 - \{0,1,\infty \} \to \G_m = \mathbb{P}^1 -\{0, \infty\}$ denote the open immersion. Let $a_2: \proj_k^1 - \{0,1,\infty \} \cong \Spec k[z, \frac{1}{z}, \frac{1}{z-1}] \to \G_m \cong \Spec k[z, \frac{1}{z}]$ be given by $a_2^* (z) = z-1$. Consider these maps as unpointed. 

\begin{lm}\label{Sai=ciwp}
Let $i=1$ or $2$. The following diagram, whose top horizontal map is the isomorphism of Proposition \ref{canonical_iso_SigmaP1minus3points}, $$\xymatrix{ & S \G_m   \ar[r]^{\cong} &\Sigma \G_m & \\ S (\proj_k^1 - \{0,1,\infty \})\ar[ru]^{S a_i} \ar[r]^{\cong}  & \Sigma (\proj_k^1 - \{0,1,\infty \})   && \ar[ll] \ar[lu]_{\Sigma c_i}  \Sigma (\G_m \vee \G_m) }$$  is commutative in $\ho_{\Aone} \Spaces$. 
\end{lm}

\begin{proof}
We keep the notation of the proof of Proposition \ref{canonical_iso_SigmaP1minus3points}. Let $i_0: \{0\} \to Z $ and $i_1: \{1\} \to Z $ be the closed (and open) immersions, and let $j_k = i \circ i_\ell$ for $\ell=0,1$. Let $\calN(j_\ell)$ denote the normal bundle to $j_\ell$. The decomposition of $Z$ as the disjoint union $Z = \{0\} \coprod \{1\}$ gives a decomposition $\calN(i) = \calN(j_0) \coprod \calN(j_1)$. The maps of pairs $(\calN(j_\ell),\calN(j_\ell) - 0) \to (\calN, \calN-0)$ for $\ell=0,1$ determine maps $\Th(\calN(j_\ell))  \to \Th(\calN(i))$ which combine to give an isomorphism $$\Th(\calN(j_0)) \vee \Th(\calN(j_1)) \to  \Th(\calN(i)).$$ Mapping $\Th(\calN(j_0))$ to the basepoint thus determines a map $\Th(\calN(i)) \to \Th(\calN(j_1))$. And we have the analogous map $\Th(\calN(i)) \to \Th(\calN(j_0))$. 

The diagram $$ \xymatrix {  \mathbb{A}^1_k/ ( \mathbb{A}^1_k- i(Z))  \ar[d]   & \ar[l] \Th(\calN (i)) \ar[d] \\   \mathbb{A}^1_k/ ( \mathbb{A}^1_k- j_0(\{0\}) & \ar[l] \Th(\calN (j_\ell))}$$ in $\ho_{\Aone}\Spaces$ is commutative by the functoriality of blow-ups and the construction of the canonical isomorphism of \cite[Theorem 2.23]{MV}.

Use the trivialization of $\calN(j_\ell)$ from the proof of Proposition \ref{canonical_iso_SigmaP1minus3points}. We obtain an isomorphism $\Th(\calN(j_\ell)) \to \proj^1$. This isomorphism fits into the commutative diagram $$\xymatrix{ \Th(\calN (i)) \ar[d] &  \ar[l] \proj^1 \vee \proj^1 \ar[d]\\ \Th(\calN (j_\ell))  & \ar[l] \proj^1 }$$where the top horizontal map is as in the proof of Proposition \ref{canonical_iso_SigmaP1minus3points}, and the right vertical morphism crushes the factor not corresponding to $\ell$. 

Place the two previous commutative diagrams side by side and use the isomorphism $\proj^1 \to \Sigma \G_m$ from the proof of Proposition \ref{canonical_iso_SigmaP1minus3points} to replace the $\proj^1$'s with $\Sigma \G_m$'s. Then note that the composition $$ \mathbb{A}^1_k/ ( \mathbb{A}^1_k- i(Z))  \to \mathbb{A}^1_k/ ( \mathbb{A}^1_k- j_\ell(\{\ell\})) \to \proj^1 \to \Sigma \G_m$$ is the composition of $S a_\ell$ with $S \G_m \to \Sigma \G_m$ after identifying $\mathbb{A}^1_k/ ( \mathbb{A}^1_k- i(Z)) \cong  \Sigma (\proj_k^1 - \{0,1,\infty \})$.  This proves the proposition. 
\end{proof}

Let $\overline{01}$ denote the tangential base point of $\proj_k^1 - \{0,1,\infty \}$ at $0$ pointing in the direction of $1$, as in \cite[\S 15]{Deligne_Galois_groups} \cite{Nakamura}, so $\overline{01}$ determines the fiber functor associated to the geometric point $$\proj_k^1 - \{0,1,\infty \}= \Spec k[z, \frac{1}{z}, \frac{1}{1-z}] \leftarrow \Spec \cup_{n\in \Z_{>0}} \kbar ((z^{1/n}))$$ $$k[z, \frac{1}{z}, \frac{1}{1-z}] \to  k(z) \to  \cup_{n\in \Z_{>0}} \kbar ((z^{1/n})).$$ Let $\pi = \pi_1^{\et}(\proj_{\overline{k}}^1 - \{0,1,\infty \}, \overline{01})$. Since the \'etale fundamental group is invariant under algebraically closed base change in characteristic $0$, we have a canonical isomorphism $\pi \cong \pi_1^{\et}(\proj_{\C}^1 - \{0,1,\infty \}, \overline{01})$. There is a canonical isomorphism between $\pi_1^{\et}(\proj_{\C}^1 - \{0,1,\infty \}, \overline{01})$ and the profinite completion of the topological fundamental group. Let $x$ be the element of the topological fundamental group represented by a small counter-clockwise loop around $0$ based at $\overline{01}$, and let $y$ be the path formed by traveling along $[0,1]$, then traveling along the image of $x$ under $z \mapsto 1-z$, and then traveling back from $1$ to $0$ along $[0,1]$. Putting this together, we have fixed an isomorphism $$\pi \cong \langle x,y \rangle^{\wedge} $$ between $\pi$ and the profinite completion of the free group on two generators $x$ and $y$. Recall that in Example \ref{Hi_spaces_functor}, we have defined $\pi' = \langle x,y \rangle^{\wedge}$ and maps out of $\pi_1(\G_{m,\overline{k}} \vee \G_{m,\overline{k}}, \ast)$ to inverse systems of finite groups factor through $\pi'$ by Lemma \ref{From_pi'}.

Let $x_n^*, y_n^* \in \Hom(\pi,\Z/n)$ be defined by $x_n^* (x) = 1$, $x_n^* (y) = 0$, $y_n^* (x) = 0$, and $y_n^* (y) = 1$. By Proposition \ref{H1=Hompi}, $\rH^1(\proj_{\kbar}^1 - \{0,1,\infty \},\Z/n)$ is a free $\Z/n$-module with basis $\{x_n^*, y_n^* \}$. Making the analogous definitions of $x_n^*$ and $y_n^*$ with $\pi'$ replacing $\pi$, Proposition \ref{H1=Hompi} and Lemma \ref{From_pi'} show that $\rH^1(\G_{m,\kbar} \vee \G_{m,\kbar}, \Z/n)$ is a free $\Z/n$-module with basis $\{x_n^*, y_n^* \}$. By Proposition \ref{Hi+1Sigma=Hi}, we obtain isomorphisms $\rH^2(\Sigma X_{\kbar},\Z/n) \cong \Z/n x_n^* \oplus \Z/n y_n^*$ for $X = \proj_{k}^1 - \{0,1,\infty \},$ and $\G_m \vee \G_m$.  

Let $\wp: \Sigma (\G_m \vee \G_m)  \to \Sigma (\proj_k^1 - \{0,1,\infty \})$ be any map determining the canonical isomorphism of Proposition \ref{canonical_iso_SigmaP1minus3points}.

\begin{pr}\label{H2wp}
$\rH^2(\wp_{\kbar}, \Z/n)$ is computed by $\rH^2(\wp_{\kbar}, \Z/n)(x_n^*) = x_n^*$ and $\rH^2(\wp_{\kbar}, \Z/n)(y_n^*) = y_n^*$.
\end{pr}

\begin{proof}

By an abuse of notation, let $\wp$ also denote the composite morphism $\Sigma (\G_m \vee \G_m) \to \Sigma (\proj_k^1 - \{0,1,\infty \}) \to S (\proj_k^1 - \{0,1,\infty \})$ in $\ho_{\Aone} \Spaces$, and identify $\rH^2(S (\proj_{\kbar}^1 - \{0,1,\infty \}),\Z/n)$ with $\rH^2(\Sigma (\proj_{\kbar}^1 - \{0,1,\infty \}),\Z/n)$ and $\rH^2(S \G_m,\Z/n)$ with $\rH^2(\Sigma \G_m,\Z/n)$ by the isomorphisms in the statement of Lemma \ref{Sai=ciwp}. Let $\Sigma a_i$ denote the composition of $S a_i$ with the canonical map $S \G_m \to \Sigma \G_m$.

Then Lemma \ref{Sai=ciwp} says that $\Sigma a_i  \circ \wp = \Sigma c_i$. The dual to the counterclockwise loop based at $1$ in $\G_m(\C)$ determines a canonical element $z_n^*$ of $\rH^2(\Sigma \G_{m,\kbar}, \Z/n)$ by the comparison between the \'etale and topological fundamental groups \cite[XII Corollaire 5.2]{sga1}, Proposition  \ref{Hi+1Sigma=Hi}, and Proposition \ref{H1=Hompi}.  By the construction of $x_n^*$ and $y_n^*$, we have that $(\Sigma a_1)^*(z_n^*) = x_n^* $, $(\Sigma a_2)^*(z_n^*) = y_n^*$, $(\Sigma c_1)^*(z_n^*) = x_n^*$ and $(\Sigma c_2)^*(z_n^*) = y_n^*$. This shows the proposition.
\end{proof}

\section{Desuspending $\Sigma (\proj_k^1 - \{0,1,\infty \})$}\label{Sec_desuspending}

We use the Galois action on $\pi_1^{\et}(\proj_{\kbar}^1 - \{0,1,\infty \})$ to show that $\proj_k^1 - \{0,1,\infty \}$ and $\G_{m,k} \vee \G_{m,k}$ are distinct desuspensions of $\Sigma (\proj_k^1 - \{0,1,\infty \})$.  Recall the definition of $\pi'$ from Example \ref{LEtGmkveeGmk}. Here are the needed facts about the Galois action on $\pi = \pi_1^{\et}(\proj_{\overline{k}}^1 - \{0,1,\infty \}, \overline{01})$.

An element $g \in G$ acts on $\pi$ by \begin{equation}\label{G-action_pi} g(x)  =  x^{\chi(g)} \quad \quad g(y)  =  \frak{f}(g)^{-1}y^{\chi(g)}\frak{f}(g)\end{equation} where $\mathfrak{f}: G \rightarrow [\pi]_2$ is a cocycle with values in  the commutator subgroup $[\pi]_2$ of $\pi$. See \cite[Proposition 1.6]{Ihara_GT}.  Since $\overline{01}$ is a rational tangential base-point, $\overline{01}$ splits the homomorphism $\pi_1^{\et}(\proj_{k}^1 - \{0,1,\infty \}, \overline{01}) \to \pi_1^{\et} \Spec k \cong G$, giving an isomorphism $\pi_1^{\et}(\proj_{k}^1 - \{0,1,\infty \}, \overline{01}) \cong  \pi \rtimes G$.

Let $\pi = [\pi]_1 \supseteq [\pi]_2  \supseteq [\pi]_3 \supseteq \ldots$ denote the lower central series of $\pi$, so $[\pi]_n$ is the closure of the subgroup generated by commutators of elements of $\pi$ with elements of $[\pi]_{n-1}$.  Use the analogous notation for the lower central series of any profinite group.

Let $\iota: \pi' \to \pi$ be the homomorphism of groups $\iota(x) = x$ and $\iota(y) = y$, and let $\iota^{\ab}: (\pi')^{\ab} \to \pi^{\ab}$ denote the induced map on abelianizations. Note that $\iota^{\ab}$ is $G$-equivariant. 

\begin{lm}\label{no_G-equivariant_extnsions_iota_ab}
Let $k$ be a number field not containing the square root of $2$. Then there is no continuous homomorphism $\pi' \times G \to \pi \rtimes G$ over $G$ inducing $\iota^{\ab} \rtimes 1_G$ after abelianization.
\end{lm}

\begin{proof}
Suppose to the contrary that $\theta$ is such a map. Because the subgroups of the lower central series are characteristic, $\theta$ induces a commutative diagram \begin{equation}\label{n_n+1_LCS_theta}\xymatrix{1 \ar[r] & [\pi]_{n}/ [\pi]_{n+1} \ar[r] &  \pi/ [\pi]_{n+1} \rtimes G \ar[r] & \pi/ [\pi]_{n} \rtimes G \ar[r] & 1\\ 1 \ar[r] & [\pi']_{n}/ [\pi']_{n+1} \ar[u]^{\overline{\theta}_n^{n+1}} \ar[r] &  \pi'/ [\pi']_{n+1}\rtimes G \ar[u]^{\overline{\theta}_{n+1}} \ar[r] & \pi'/ [\pi']_{n}\rtimes G \ar[u]^{\overline{\theta}_{n}} \ar[r] & 1}.\end{equation} Thus if $\overline{\theta}_n^{n+1}$ and $\overline{\theta}_{n}$ are isomorphisms, so is $\overline{\theta}_{n+1}$. Since $\pi'$ and $\pi$ are isomorphic to the profinite completion of the free group on two generators, $[\pi']_n/[\pi']_{n+1}$ and  $[\pi]_n/[\pi]_{n+1}$ are isomorphic to the degree $n$ graded component of the free Lie algebra on the same generators over $\hatZ$.  Since $\overline{\theta}_{2} = \iota^{\ab} \rtimes 1_G$ is an isomorphism, it follows that $\overline{\theta}_n^{n+1}$ is an isomorphism. By induction, it follows that $\overline{\theta}_{n}$ is an isomorphism for all $n$. 

The extension $$  1 \to [\pi]_{n}/ [\pi]_{n+1} \to  \pi/ [\pi]_{n+1} \rtimes G \to \pi/ [\pi]_{n} \rtimes G \to 1$$ is classified by the element of $\rH^2(\pi/ [\pi]_{n}  \rtimes G,[\pi]_{n}/ [\pi]_{n+1})$ represented by the inhomogeneous cocycle $\varphi_{n}$ $$
\varphi_{n}(\gamma \rtimes g, \eta \rtimes h) =  s(\gamma) g s(\eta) s(\gamma g \eta)^{-1}
$$ where $s: \pi/ [\pi]_{n} \to \pi/ [\pi]_{n+1}$ is a continuous set-theoretic section of the quotient map $\pi/ [\pi]_{n+1} \to \pi/ [\pi]_{n}$.  See for example \cite[IV 3]{Brown_coh_groups}. Let $\varphi_{n}'$ denote the analogous inhomogeneous cocycle obtained by replacing $\pi$ with $\pi'$.

This association of a class in $\rH^2(\pi/[\pi]_n \rtimes G, [\pi]_n/[\pi]_{n+1})$ to an extension of $\pi/[\pi]_n \times G$ by $[\pi]_n/[\pi]_{n+1}$  induces a bijection between $\rH^2(\pi/[\pi]_n \rtimes G, [\pi]_n/[\pi]_{n+1})$ and isomorphism classes of extensions \cite[IV Theorem 3.12]{Brown_coh_groups}. Since $\overline{\theta}_{n}$ is an isomorphism, it follows that $\varphi_{n}'$ and $(\overline{\theta}_{n})^* \varphi_{n}$ represent the same class in  $\rH^2(\pi/[\pi]_n \rtimes G, [\pi]_n/[\pi]_{n+1})$. 

By \eqref{G-action_pi} and \eqref{G-action_pi'}, $$\pi/[\pi]_{2} \cong \hatZ(1)x \oplus \hatZ(1) y$$ $$[\pi]_2/[\pi]_{3} \cong \hatZ(2) [x,y]$$ \begin{equation}\label{pi34}[\pi]_3/[\pi]_{4} \cong \hatZ(3) [[x,y],x] \oplus \hatZ(3) [[x,y],y],\end{equation} and the same isomorphisms hold with $\pi'$ replacing $\pi$.

We claim that $\overline{\theta}_{3}$ is given by \begin{align}\label{theta3}\overline{\theta}_{3}(x \rtimes 1) = x \rtimes 1 \qquad \overline{\theta}_{3}(y \rtimes 1) = y \rtimes 1\qquad \overline{\theta}_{3}(1\rtimes g) = [x,y]^{c(g)}\rtimes g\end{align} for all $g \in G$, where $$c: G \to \hatZ(2) $$ is a coycle. To see this, note that the hypothesis on $\overline{\theta}_{2}$ implies that $\overline{\theta}_{3} (x \rtimes 1) = x  [x,y]^{c_1(g)} \rtimes 1$ with $c_1(g)$ in $\hatZ$. Similarly, $\overline{\theta}_{3} (1 \rtimes g) =  [x,y]^{c(g)} \rtimes g$ with $c(g)$ in $\hatZ$. Since $\theta$ is a homomorphism, we have  $\overline{\theta}_{3}(g x) = \overline{\theta}_{3}(g) \overline{\theta}_{3}(x)$. Since $g x = x^{\chi(g)} \rtimes g$ and $[x,y]$ is in the center, we have $$\overline{\theta}_{3}(g x) = \overline{\theta}_{3}(x)^{\chi(g)}\ott(g) = x^{\chi(g)}  [x,y]^{c_1(g) \chi(g)+c(g)} \rtimes g.$$ On the other hand, $$\ott(g) \ott(x) = ([x,y]^{c(g)} \rtimes g) (x [x,y]^{c_1(g)} \rtimes 1) = x^{\chi(g)}  [x,y]^{c_1(g) \chi(g)^2+c(g)} \rtimes g.$$ Thus $c_1(g) \chi(g)^2+c(g) = c_1(g) \chi(g)+c(g)$ for all $g$ in $G$. It follows that $c_1(g) = 0$.  Since $\frak{f}(g)$ is in $[\pi]_2$ and elements of $[\pi]_2$ are all in the center of $\pi/[\pi]_3$, switching $x$ and $y$ induces an isomorphism on $\pi/[\pi]_3$. The same argument therefore implies that $\ott(y) = y$. Since $\ott$ is a homomorphism when restricted to $1 \rtimes G$, it follows that $c$ is a cocycle, showing \eqref{theta3}.

By \eqref{pi34}, we have a direct sum decomposition $$\rH^2(\pi/[\pi]_3 \rtimes G, [\pi]_{3}/[\pi]_4) \cong \rH^2(\pi/[\pi]_3 \rtimes G, \hatZ(3))[[x,y],x] \oplus \rH^2(\pi/[\pi]_3 \rtimes G, \hatZ(3) ) [[x,y],y].$$ Define $\varphi_{3, [[x,y],x]}$ and $\varphi_{3,[[x,y],y]}$ so that under this isomorphism $\varphi_ 3$ decomposes as $\varphi_3 = \varphi_{3, [[x,y],x]} \oplus\varphi_{3,[[x,y],y]}$. It was calculated in \cite{PIA} that $\varphi_{3,[[x,y],x]}$ is represented by the cocycle maping  $(y^{a_1} x^{b_1} [x,y]^{c_1} \rtimes g_1, y^{a_2} x^{b_2} [x,y]^{c_2} \rtimes g_2) $ to
\begin{align*}
c_1\chi(g_1) b_2 +  {b_1 + 1\choose 2} \chi(g_1) a_2 + b_1 \chi(g_1)^2 a_2 b_2 -  \frac{\chi(g_1) -1}{2} \chi(g_1)^2 c_2 
 \end{align*} and that $\varphi_{3,[[x,y],y]}$ is represented by the cocycle maping  $(y^{a_1} x^{b_1} [x,y]^{c_1} \rtimes g_1, y^{a_2} x^{b_2} [x,y]^{c_2} \rtimes g_2) $ to \begin{align*}
 c_1\chi(g_1) a_2 + b_1 {\chi(g_1)a_2 + 1\choose 2} -  \chi(g_1){\chi(g_1) \choose 2} c_2 - f(g_1) \chi(g_1)a_2
 \end{align*} where $f: G \to \hatZ(2)$ is such that $\frak{f}(g) = [x,y]^{f(g)} $ in $\pi/[\pi]_3$.
 
We may similarly decompose $\varphi_3'$ as $\varphi'_3 = \varphi'_{3, [[x,y],x]} \oplus \varphi'_{3,[[x,y],y]}$. By the above calculation of $\ott$, and the expressions \eqref{G-action_pi} and \eqref{G-action_pi'} for the $G$-action on $\pi$ and $\pi'$, we have that $\varphi'_{3, [[x,y],x]} $ and   $\varphi'_{3,[[x,y],y]}$ are obtained from the expressions for $ \varphi_{3, [[x,y],x]} $ and $\varphi_{3,[[x,y],y]}$ by setting $f=0$. 
 
It follows that $\varphi_3' - (\ott)^* \varphi_3$ is represented by the direct sum of two cocycles, given by sending $(y^{a_1} x^{b_1} [x,y]^{c_1} \rtimes g_1, y^{a_2} x^{b_2} [x,y]^{c_2} \rtimes g_2)$ to  $$ (-c(g_1) \chi(g_1) b_2  +\frac{\chi(g_1) -1}{2} \chi(g_1)^2 c(g_2) )[[x,y],x]$$ and  $$ (-c(g_1)\chi(g_1) a_2 +   \chi(g_1){\chi(g_1) \choose 2} c(g_2) + f(g_1) \chi(g_1)a_2)[[x,y],y]$$ respectively. 

Using the above direct sum decomposition of $\rH^2(\pi/[\pi]_3 \rtimes G, [\pi]_3/[\pi]_4)$, this implies that $$\varphi_{3,[[x,y],x]}' - (\ott)^* \varphi_{3,[[x,y],x]} = -c \cup b + \frac{\chi(g) -1}{2} \cup c$$ and

$$\varphi_{3,[[x,y],y]}' - (\ott)^* \varphi_{3,[[x,y],y]} = -c \cup a + \frac{\chi(g) -1}{2} \cup c + f \cup a, $$ where these equalities are in $ \rH^2(\pi'/[\pi']_3 \rtimes G,\hatZ(3)),$ and where $f: G \to \hatZ(2)$ is considered via pullback as an element of $\rH^1(\pi'/[\pi']_3 \rtimes G, \hatZ(2))$, $a: \pi'/[\pi']_3 \rtimes G \to \hatZ(1)$ is the cocyle $y^{a} x^{b} [x,y]^{c} \rtimes g \mapsto a$, $b$ is defined similarly, and $\frac{\chi(g) -1}{2}$ is the cocycle $g \mapsto \frac{\chi(g) -1}{2}$ taking values in $\hatZ(1)$ pulled back to $\pi'/[\pi']_3 \rtimes G$.

As shown above, the existence of $\theta$ therefore implies that $-c \cup b + \frac{\chi(g) -1}{2} \cup c = 0$ and $ -c \cup a + \frac{\chi(g) -1}{2} \cup c + f \cup a = 0$ in $\rH^2(\pi'/[\pi']_3 \rtimes G,\hatZ(3))$. 

Consider first the equality $-c \cup b + \frac{\chi(g) -1}{2} \cup c = 0$. Since the cup product is graded-commutative, we may rewrite this equality as $(b  + \frac{\chi(g) -1}{2}) \cup c  = 0 $. The quotient map $\hatZ(3) \to \Z/2$ determines map $\rH^2(\pi'/[\pi']_3 \rtimes G, \hatZ(3)) \to \rH^2(\pi'/[\pi']_3 \rtimes G, \Z/2)$. Passing to the image under this map, we have an equality $\overline{(b  + \frac{\chi(g) -1}{2})} \cup \overline{c} = 0$, where $\overline{c}$ denotes the image of $c$ and $\overline{(b  + \frac{\chi(g) -1}{2})}$ denotes the image $b  + \frac{\chi(g) -1}{2}$.  Recall that for any $\beta \in k^*$ with a chosen compatible system of $n$th roots of $\beta$, there is a Kummer cocycle $\kappa(\beta): G \to \Zhat(1) \cong \varprojlim_n \mu_{n}(\kbar)$ defined by $g \sqrt[n]{\beta} = {\kappa(b)(g)}_n \sqrt[n]{\beta} $ where ${\kappa(b)(g)}_n$ is the element of $\mu_n(\kbar)$ determined by $\kappa(b)(g)$. We may define a homomorphism $G \to \pi'/[\pi']_3 \rtimes G$ by $g \mapsto y^{\kappa(\beta)}\rtimes g$. Pulling back the equality $\overline{(b  + \frac{\chi(g) -1}{2})} \cup \overline{c} = 0$ by this homomorphism gives the equality $(\kappa(b) + \overline{\frac{\chi(g) -1}{2}}) \cup \overline{c} = 0$ in $\rH^2(G, \Z/2) $ because $c$ and $\frac{\chi(g) -1}{2}$ are pulled back from $G$. Note that any element of $\rH^1(G, \Z/2)$ is of the form $(\kappa(b) + \frac{\chi(g) -1}{2})$ for an appropriate choice of $\beta$. By the non-degeneracy of the cup product $\rH^1(G, \Z/2) \otimes \rH^1(G, \Z/2) \to \rH^2(G, \Z/2)$, it follows that $\overline{c} \in \rH^2(G, \Z/2)$ is zero.

Consider now the second equality $-c \cup a + \frac{\chi(g) -1}{2} \cup c + f \cup a = 0$ in $\rH^2(\pi'/[\pi']_3 \rtimes G,\hatZ(3))$, and again pass to the image under $\rH^2(\pi'/[\pi']_3 \rtimes G, \hatZ(3)) \to \rH^2(\pi'/[\pi']_3 \rtimes G, \Z/2)$. Since $\overline{c} = 0$ in $\rH^2(G, \Z/2)$, we have $ \overline{f} \cup \overline{a} = 0$ in $\rH^2(\pi'/[\pi']_3 \rtimes G,\hatZ(3))$. On the other hand, $f: G \to \hatZ(2)$ is known by work of Ihara \cite[6.3 Thm~p.115]{Ihara_Braids_Gal_grps} \cite{IKY}, Anderson  \cite{Anderson_hyperadelic_gamma}, and Coleman \cite{Coleman_Gauss_sum}, and we can show that this is inconsistent with $\overline{f} \cup \overline{a} =0$ in the following manner. Namely, $f(g) = \frac{1}{24}(\chi(g)^2 -1)$. See \cite[12.5.2]{PIA}. By \cite[Lemma 31]{PIA}, the image of $f$ under the map $\rH^1(G, \hatZ(2)) \to \rH^1(G, \Z/2) \cong k^*/(k^*)^2$ is represented by $2 \in k^*$. For any $\alpha \in k^*$ we may choose a compatible system of $n$th roots of $\alpha$ and define a homomorphism $G \to \pi'/[\pi']_3 \rtimes G$ by $g \mapsto y^{\kappa(\alpha)}\rtimes g$. Pulling back $\overline{f} \cup \overline{a}$ by this homomorphism gives $\overline{f} \cup \kappa(\alpha) \in \rH^2(G, \Z/2)$. Thus $\kappa(2) \cup \kappa(\alpha) = 0 \in \rH^2(G, \Z/2)$ for all $\alpha \in k^*$. Since $k$ does not contain the square root of $2$, this contradicts the nondegeneracy of the the cup product, giving the desired contradiction.
\end{proof}

\begin{tm}\label{no_desuspension}
Let $k$ be a finite extension of $\Q$ not containing the square root of $2$. There is no morphism $\rho: \G_{m,k} \vee \G_{m,k} \to \proj_{k}^1 - \{0,1,\infty \}$ in $\ho_{\Aone} \Spaces$ such that $\Sigma \rho =\wp$ in $\ho_{\Aone} \Spaces$.
\end{tm}

\begin{proof}
Suppose to the contrary that we have such a morphism $\rho$. The geometric point $\overline{01}$ of $\proj_{k}^1 - \{0,1,\infty \}$ and the extension of the $k$-basepoint of $\G_{m,k} \vee \G_{m,k}$ to a geometric point allow us to consider $(\proj_{k}^1 - \{0,1,\infty \}, \overline{01})$ and $(\G_{m,k} \vee \G_{m,k}, \ast)$ as objects of $\Spaces^+$. Since $\proj_{k}^1 - \{0,1,\infty \}$ and $\G_{m,k} \vee \G_{m,k}$ have connected \'etale homotopy type, $\rho$ is a morphism in $\Spaces^+_c$. Thus $\rho$ induces an outer continuous homomorphism $\rho_*: \pi' \rtimes G \to \pi \rtimes G$ by   Proposition \ref{etpi_on_spaces}, Lemma \ref{From_pi'}, and taking the inverse limit. We may choose a continuous homomorphism over $G$ representing $\rho_*$. By a slight abuse of notation, we call this representative $\rho_*$ as well.

Let $(\rho_{\kbar})_{\ast}$ denote the induced map $\pi' \to \pi$. It follows from Proposition \ref{H1=Hompi} that the induced map $\rho^*: \rH^1(\proj_{\kbar}^1 - \{0,1,\infty \} ,\Z/n) \to \rH^1(\G_{m,\kbar} \vee \G_{m,\kbar},\Z/n) $ is computed $\rho^* = \Hom((\rho_{\kbar})_{\ast},\Z/n)$. By Proposition \ref{Hi+1Sigma=Hi}, $\rH^2(\wp_{\kbar}) = \rH^1(\rho_{\kbar})$. Combining the two previous, we have $\rH^2(\wp_{\kbar}) = \Hom((\rho_{\kbar})_{\ast},\Z/n)$. By Proposition \ref{H2wp}, it follows that $\Hom(\rho_*,\Z/n)(x_n^*) = x_n^*$ and $\Hom(\rho_*,\Z/n)(y_n^*) = y_n^*$. Since $n$ is arbitrary, it follows that  $(\rho_{\kbar})_{\ast}^{\ab} = \iota^{\ab}$. 

We claim that after modifying $\rho_{\ast}$ by an inner automorphism, the map $\rho_{\ast}^{\ab}: \pi'/[\pi']_2 \rtimes G \to  \pi/[\pi]_2 \rtimes G $ induced by $\rho_{\ast}$ is $\iota^{\ab} \rtimes 1_G$. Note the commutative diagram $$\xymatrix{ 1 \ar[r] &  \pi/[\pi]_2 \ar[r] &  \pi/[\pi]_2 \rtimes G  \ar[r] & G \ar[r] & 1 \\ 1 \ar[r] &  \pi'/[\pi']_2 \ar[u]^{(\rho_{\kbar})_{\ast}^{\ab}} \ar[r] &  \pi'/[\pi']_2 \rtimes G \ar[u]^{\rho_{\ast}^{\ab}}  \ar[r] & G \ar[r] \ar[u]^{1_G} & 1.}$$ Since $1_G$ and $(\rho_{\kbar})_{\ast}^{\ab}$ are isomorphisms, so is $\rho_{\ast}^{\ab}$. It follows by induction that $\overline{(\rho_{\ast})}_n: \pi'/[\pi']_n \rtimes G \to \pi/[\pi]_n \rtimes G$ is an isomorphism, cf. \eqref{n_n+1_LCS_theta}. 

Let $\varphi_2 \in \rH^2(\pi/[\pi]_2 \rtimes G, [\pi]_2/[\pi]_3)$ be the element classifying $$1 \to [\pi]_2/[\pi]_3 \to \pi/[\pi]_3 \rtimes G \to \pi/[\pi]_2 \rtimes G \to 1,$$ and define $\varphi_2'$ by replacing $\pi$ with $\pi'$ in the definition of $\varphi_2$. Since  $\overline{(\rho_{\ast})}_3$ is an isomorphism, $\overline{(\rho_{\ast})}_2^* \varphi_2 = \varphi_2'$. By \cite[Proposition 7]{PIA}, $\varphi_2 = b \cup a$, where $b: \pi/[\pi]_2 \rtimes G \to \hatZ(1)$ is the cocycle $y^a x^b \rtimes g \mapsto b$ and $a: \pi/[\pi]_2 \rtimes G \to \hatZ(1)$ is the cocycle $y^a x^b \rtimes g \mapsto a$. Since conjugation by $\mathfrak{f}(g)$ is trivial in $\pi/[\pi]_3$, it follows that $\varphi_2' = b \cup a$, where $a$ and $b$ are defined by replacing $\pi'$ with $\pi$ in the previous.  Because $(\rho_{\kbar})_{\ast}^{\ab} = \iota^{\ab}$, we have $\overline{(\rho_{\ast})}_2 (y^a x^b \rtimes g) = y^{a + \alpha(g)} x^{b + \beta(g)} \rtimes g$ where $\alpha,\beta: G \to \hatZ(1)$ are cocycles. Thus $\overline{(\rho_{\ast})}_2^* \varphi_2 = (b + \beta) \cup (a + \alpha)$. Thus $(b + \beta) \cup (a + \alpha) = b \cup a$. Since the cup product is non-degenerate, it  follows that $\beta$ and $\alpha$ are trivial in cohomology. Thus after modifying $\rho$ by an inner automorphism, we may assume $\rho_{\ast}^{\ab}= \iota^{\ab} \rtimes 1_G$. This contradicts Lemma \ref{no_G-equivariant_extnsions_iota_ab}. 
\end{proof}

} 

\bibliographystyle{SP1mD}


\bibliography{SP1minus_desuspension}


\end{document}